\documentclass{amsart}

\input{amssym.def}
\input{amssym.tex}

\usepackage[curve,tips]{xypic}
\SelectTips{eu}{11}
\UseTips

\usepackage{amsthm}
\usepackage{epsfig}
\usepackage{color}

\title[Nilpotent completions, Grothendieck pairs, and problems of Baumslag]{Nilpotent completions of groups, Grothendieck pairs, and four problems of Baumslag}
\author{M. R. Bridson}
\author{A. W. Reid}
\address{\newline Mathematical Institute, 
\newline Andrew Wiles Building,
\newline University of Oxford,
\newline Oxford OX2 6GG, UK}
\email{ bridson@maths.ox.ac.uk}
\address{\newline Department of Mathematics,
\newline University of Texas, 
\newline Austin, TX 78712, USA}
\email{ areid@math.utexas.edu}
\thanks{Bridson was supported in part by grants from the
EPSRC of Great Britain and a Royal Society Wolfson Merit Award. Reid was supported in part 
by an NSF grant}

\def\N{\mathbb{N}}
\def\epi{\twoheadrightarrow}
\def\wh{\widehat}
\def\g{\gamma}
\def\G{\Gamma}

\def\Z{\Bbb Z}
\def\Q{\Bbb Q}
\def\R{\Bbb R}
\def\F{\Bbb F}

\def\ker{\rm{ker}}

\def\nil{\rm{nil}}

\def\PSL{\rm{PSL}}

\def\Epi{\rm{Epi}}

\def\qed{ $\sqcup\!\!\!\!\sqcap$}

\def\G{\Gamma}

\def\fn{{\rm{fn}}} 
\def\<{\langle}
\def\>{\rangle}
\def\ilim{\varprojlim}

\def\Hom{{\rm{Hom}}}
\def\Epi{{\rm{Epi}}}

\newtheorem{theorem}{Theorem}[section]
\newtheorem{lemma}[theorem]{Lemma}
\newtheorem{corollary}[theorem]{Corollary}
\newtheorem{proposition}[theorem]{Proposition}
\newtheorem{problem}[theorem]{Problem}
\newtheorem{letterthm}{Theorem}

\theoremstyle{definition} 

\newtheorem{remark}[theorem]{Remark}
\newtheorem{example}[theorem]{Example}
\newtheorem{remarks}[theorem]{Remarks}

\begin{document}

\begin{abstract} Two groups are said to have the same nilpotent genus if they have the same nilpotent quotients.
We answer four 
questions of Baumslag concerning nilpotent completions.
(i) There exists a pair of finitely generated, residually
torsion-free-nilpotent groups of the same nilpotent genus such that
one is finitely presented and the other is not.  
(ii) There exists a pair of finitely presented, residually
torsion-free-nilpotent groups of the same nilpotent genus such that
one has a solvable conjugacy problem and the other does not. (iii) There exists a pair of finitely generated, residually
torsion-free-nilpotent groups of the same nilpotent genus such that
one has finitely generated second homology $H_2(-,\Z)$ and the other does not.
 (iv) A non-trivial normal subgroup of infinite index in a finitely generated parafree group cannot be finitely generated. In proving
this last result, we establish that the first $L^2$-Betti number of a finitely generated parafree group in the same nilpotent genus as
a free group of rank $r$ is $r-1$. It
follows that the reduced $C^*$-algebra of the group is simple if $r\ge 2$, and that a version of the Freiheitssatz holds for
parafree groups.
\end{abstract}



\maketitle

%
%
%
%
\section{Introduction}
If each finite subset of a group $\G$ injects into some nilpotent (or finite) quotient of $\G$,
then it is reasonable to expect that one will be able to detect many properties of $\G$ from
the totality of its nilpotent (or finite) quotients. Attempts to lend precision to this observation,
and to test its limitations, have surfaced repeatedly in the study of discrete and profinite groups
over the last forty years, and there has been a particular resurgence of interest recently,
marked by several notable breakthroughs. We take up this theme here, with a focus on the nilpotent completions
of residually torsion-free-nilpotent groups.

We begin by recalling some terminology.
A group $\G$ is said to be {\em residually 
nilpotent} (resp. {\em residually torsion-free-nilpotent})
if for each non-trivial $\gamma\in\G$ there exists a nilpotent
group (resp. torsion-free-nilpotent group) $Q$ and a homomorphism
$\phi:\G\rightarrow Q$ with $\phi(\gamma)\neq 1$. Thus 
$\G$ is residually nilpotent if and only if  $\bigcap \G_n = 1$, where $\G_n$ is the
$n$-th term of the {\em lower central series} of $\G$, defined inductively by
setting $\G_1=\G$ and defining $\G_{n+1} = \< [x,y] : x\in \G_n, y\in \G\>.$ 

We say that two residually nilpotent groups $\G$ and $\Lambda$ have the
same {\em nilpotent genus}\footnote{Baumslag uses the simpler term
  ``genus", but this conflicts with the usage of the term in the study
of profinite groups, and since we study different completions it
seems best to be more precise here.} if they have the
same lower central series quotients; i.e. $\G/\G_c \cong \Lambda/\Lambda_c$ for all
$c\geq 1$.

Residually nilpotent groups with the same nilpotent genus as a
free group are termed {\em parafree}.
In \cite{baum} Gilbert Baumslag surveyed the state of the art
concerning groups of the same nilpotent genus
with particular emphasis on the nature of parafree groups. He
concludes by listing a number of open problems that are of particular
importance in the field. We will address Problems 2, 4 and 6 from his list \cite{baum},
two of which he raised again in \cite{chat} where he 
emphasised the importance of the third of the problems described below. 

\begin{problem} 
\label{baumslag1} 
Does there exist a pair of finitely  generated,
residually torsion-free-nilpotent groups of the same nilpotent genus such that
one is finitely presented and the other is not?
\end{problem}

\begin{problem} 
\label{baumslag2}
Does there exist a pair of finitely presented
residually torsion-free-nilpotent groups of the same nilpotent genus such that
one has a solvable conjugacy problem and the other does not?
\end{problem}

\begin{problem} \label{baumslag3} 
Does there exist a pair of finitely  generated, residually
torsion-free-nilpotent groups of the same nilpotent genus such that one
has finitely generated second homology $H_2(-,\Z)$ and the other does not? 
\end{problem}

\begin{problem} 
\label{baumslag4} Let $G$ be a finitely generated parafree group and
let $N<G$ be a finitely generated, non-trivial, normal
subgroup. Must $N$ be of finite index in $G$?
\end{problem}

In \cite{baum} Baumslag gives some evidence to suggest
that pairs of groups as in Problem \ref{baumslag2} might exist; he had
earlier proved that there exist finitely presented residually torsion-free
nilpotent groups with unsolvable conjugacy problem \cite{newbaum}.
Baumslag also proves a partial result in connection with Problem
\ref{baumslag4} (see Theorem 7 of \cite{baum}).  Note that Problem
\ref{baumslag4} is well known to have a positive answer for free
groups.

In this paper we answer these questions. Concerning the possible divergence in behaviour between
groups in the same nilpotent genus, we prove:

\begin{letterthm}\label{main}
\begin{enumerate}
\item There exist pairs of finitely generated, residually
torsion-free-nilpotent groups $H\hookrightarrow D$ of the same nilpotent genus  such that
$D$ is finitely presented and $H$ is not.

\item There exist pairs of finitely presented, residually
torsion-free-nilpotent groups $P\hookrightarrow\G$ of the same nilpotent genus  such that
$\G$ has a solvable conjugacy problem and $P$ does not.

\item There exist pairs of finitely generated, residually 
torsion-free-nilpotent groups $N\hookrightarrow\G$ of the same nilpotent genus  such that
$H_2(\G,\Z)$ is finitely generated but  $H_2(N,\Z)$ is not.
\end{enumerate}
\end{letterthm}

The following result strengthens items (1) and (3) of the above theorem.

\begin{letterthm} \label{t:H2}
There exist pairs of finitely generated residually
torsion-free-nilpotent groups $N\hookrightarrow \G$ that have the same nilpotent genus and the same profinite
completion, but  $\G$ is finitely presented while $H_2(N,\mathbb Q)$ is infinite dimensional.
\end{letterthm}

In the above theorems, the pairs of discrete groups that we construct  are such that the inclusion map 
induces isomorphisms of both the  profinite and pro-nilpotent completions.
These results emphasise how divergent the behaviour can be within a nilpotent genus. Our solution to
Problem \ref{baumslag4}, in contrast, establishes a commonality among parafree groups, and further commonalities
are established in Section 8.

\begin{letterthm}\label{t:N} Let $G$ be a finitely generated parafree group and let
$N<G$ be a non-trivial normal subgroup. If $N$ is finitely generated, then
$G/N$ is finite.
\end{letterthm}

Theorem \ref{t:N} is proved in Section 7.
As with our other results, the proof exploits the theory of profinite groups.
It also relies on results concerning the $L^2$-Betti numbers of discrete groups. A key observation here is
that the first $L^2$-Betti number of a finitely presented residually torsion-free-nilpotent group is an invariant of its
pro-$p$ completion for an arbitrary prime $p$ (see Corollary \ref{c:sameL2}). Theorem \ref{t:N} will emerge as a
special case of a  result concerning $L^2$-Betti numbers for dense subgroups of free pro-$p$ groups (Theorem \ref{t:general}).

In Section 8 we discuss some implications of our results. The
non-vanishing of the first $L^2$-Betti number for a parafree group
in the same nilpotent genus as a free group of 
rank $r\ge 2$ is shown to have several important consequences, two of
which are gathered in the following theorem (the second part of which
was originally proved by Baumslag \cite{Baum69}). We also apply our
results to the study of homology boundary links and prove that a
non-free parafree group cannot be isomorphic to a lattice in a
connected Lie group.

\begin{letterthm}\label{t:new} Let $G$ be a finitely generated parafree group
in the same nilpotent genus as a free group of rank
$r\ge 2$. Then the reduced $C^*$-algebra of $G$ is simple,
and every generating set for $G$  contains an $r$-element
subset that generates a free subgroup of rank $r$.
\end{letterthm}

The first part of the paper is organised as follows. In Section 2 we recall some basic properties about profinite and pro-nilpotent
completions of discrete groups and about the correspondence between the subgroup structure of the discrete
group and that of its various completions. In Section 3 we describe criteria for constructing distinct groups of the same
nilpotent genus. In particular we prove that if a map of finitely
generated discrete groups $P\hookrightarrow\G$ induces an isomorphism
of profinite completions, then $P$ and $\G$ have the same nilpotent genus. (If ones assumes merely
that $\wh{P}\cong\wh{\G}$, then the genus need not be the same.)
Our proof of Theorem \ref{main} involves the construction of carefully crafted pairs of residually torsion-free-nilpotent
groups $u:P\hookrightarrow\G$ such that $\wh{u}:\wh{P}\to\wh{\G}$ is an isomorphism. Theorem \ref{main}(1) is proved
in Section 5, and Theorem \ref{main}(2) is proved in Section 4. The proof of Theorem \ref{t:H2} is more elaborate: it
involves the construction of a finitely presented group with particular properties that may be of independent
interest (Proposition \ref{p:newD}). 

In relation to Theorem A(1), we draw the
reader's attention to a recent paper of
 Alexander Lubotzky  \cite{alex} in which he constructs pairs of groups
that have  isomorphic profinite completions but have
finiteness lengths that can be chosen arbitrarily.
These examples are $S$-arithmetic groups over function fields; they
do not have infinite proper quotients and therefore are not residually
torsion-free nilpotent, but they are residually finite-nilpotent.

Commenting on an earlier version of this manuscript, Chuck Miller pointed out
that a non-constructive solution to Problem 1.2 is implicit in
 his work with Baumslag on the isomorphism problem
for residually torsion-free nilpotent groups \cite{BaM}; see Remark \ref{r:Chuck1}. He also suggested an alternative 
proof of Theorem B; see Remark \ref{r:Chuck2}. We thank him for these comments. We also thank an anonymous referee whose comments improved Section 3.

\bigskip

\noindent{{\bf{Acknowledgement.}}} The authors thank the Mittag--Leffler Institute (Djursholm, Sweden) for
its hospitality during the drafting of this paper.

\section{Profinite and Pro-Nilpotent Completions} 
Let $\G$ be a finitely generated group. If one orders the normal
subgroups of finite index $N<\G$ 
by reverse inclusion, then the quotients 
$\G/N$ form an inverse system whose inverse limit
$$
\wh{\G} = \ilim \G/N$$
is the  {\em profinite
completion} of $\G$. Similarly, the {\em pro-nilpotent completion}, denoted $\widehat\G_{\nil}$ 
is the inverse limit of the nilpotent quotients of $\G$, and the 
{\em pro-(finite nilpotent)} completion, denoted $\widehat\G_{\fn}$, is the
inverse limit of the finite nilpotent quotients of $\G$. Given a prime $p$, the
the {\em pro-$p$ completion} $\wh{\G}_p$ is the inverse limit of the finite $p$-group
quotients of $\G$.

\begin{remarks}
\begin{enumerate}
\item The inverse limit 
topology on $\widehat\G$, on $\wh{\G}_p$, and on
$\widehat{\G}_{\fn}$ makes them {\em{compact}} topological groups. 
The induced topologies on $\G$ are called
the profinite, pro-$p$ and {\em pro-(finite nilpotent)} topologies, respectively.
\item We do not view
$\widehat{\G}_{\nil}$ as a topological group, and the absence of a 
useful topology makes it a less interesting object.
\item By construction, $\wh{\G}$ (resp.  $\wh{\G}_{\nil}$) is residually finite (resp. nilpotent), while $\widehat{\G}_{\fn}$ is residually
finite-nilpotent, and $\widehat{\G}_{p}$ is residually $p$.
\item 
The natural homomorphism $\G\to\widehat\G$ is injective if and only if
$\G$ is residually finite, while the natural maps
$\G\to\widehat\G_{\nil}$ and $\G\to\widehat\G_{\fn}$ are injective if
and only if $\G$ is residually nilpotent.
\item $\G\to\wh{\G}_{\nil}$ is an isomorphism if and only if $\G$ is nilpotent.
\item The quotients $\G/\G_c$ of $\G$
by the terms of its lower central series  are cofinal in the system of
all nilpotent quotients, so one can equally define $\widehat\G_{\nil}$ to be the
inverse limit of these.
\item If $\wh{\G}_{\nil}\cong\wh{\Lambda}_{\nil}$ then $\wh{\G}_{\fn}\cong\wh{\Lambda}_{\fn}$, but the converse is false in general since
there exist finitely generated non-isomorphic nilpotent groups that have the same finite quotients \cite{rem}.

\item Less obviously, if $\G$ and $\Lambda$ are finitely generated and
  have the same nilpotent genus then $\wh{\G}_{\fn}\cong
  \wh{\Lambda}_{\fn}$. To see this note that having the same nilpotent genus is
  equivalent to the statement that $\G$ and $\Lambda$ have the same
  nilpotent quotients, in particular the same finite nilpotent
  quotients. The conclusion $\wh{\G}_{\fn}\cong \wh{\Lambda}_{\fn}$
  follows from Theorem 3.2.7 of \cite{RZ}.

\item Finite $p$-groups are nilpotent, so for every prime $p$ there is
  a natural epimorphism $\wh{\G}_{\fn}\epi \wh{\G}_{p}$.  
 Finitely generated groups of the same nilpotent genus have
  isomorphic pro-$p$ completions for all primes $p$.
\item Every homomorphism of discrete groups $u:\G\to\Lambda$ induces maps $\wh{u}:\wh{\G}\to\wh{\Lambda}$
and $\wh{u}_{(p)}:\wh{\G}_{p}\to\wh{\Lambda}_{p}$
and $\wh{u}_{\nil}:\wh{\G}_{\nil}\to\wh{\Lambda}_{\nil}$ and $\wh{u}_{\fn}:\wh{\G}_{\fn}\to\wh{\Lambda}_{\fn}$.
\end{enumerate}
\end{remarks}

The image of the canonical map $\G\to\wh{\G}$ is dense regardless of whether $\G$ is residually
finite or not, so the restriction to $\G$ of any continuous epimomorphism from $\widehat{\Gamma}$
to a  finite group is onto. A deep theorem of Nikolov and Segal
\cite{NS} implies that if $\G$ is finitely generated then
{\em every} homomorphism from $\widehat{\G}$ to a finite group is continuous.
And the universal property of $\wh{\G}$ ensures that every homomorphism from $\G$ to a finite
group extends uniquely to $\wh{\G}$. Thus
we have the following basic result in which $\Hom(\G,Q)$ denotes the set of 
homomorphisms from the group $\G$ to the group $Q$, and $\Epi(\G,Q)$ denotes the set of epimorphisms.
If one replaces $\wh{\G}$ by $\wh{\G}_{\nil}$ or $\wh{\G}_{\fn}$ in the following lemma, one obtains bijections for finite nilpotent groups $Q$;
and for $\wh{\G}_p$ one obtains bijections when $Q$ is a finite $p$-group.

\begin{lemma}\label{NS} Let $\G$ be a finitely generated group and let  $\iota:\G\to\widehat{\Gamma}$
be the natural map to its profinite completion. Then, for every finite group $Q$,
the map $\Hom(\widehat{\Gamma},Q)\to \Hom(\G,Q)$ defined by $g\mapsto g\circ\iota$ is a bijection,
and this restricts to a bijection $\Epi(\widehat{\Gamma},Q)\to \Epi(\G,Q)$. \end{lemma}

Closely related to this, we have the following basic but important fact relating
the subgroup structures of $\G$ and $\widehat\G_{\fn}$ and $\wh{\G}_p$
(see \cite{RZ} Proposition 3.2.2, and note that the argument is valid
for other profinite completions).\\[\baselineskip]
\noindent{\bf{Notation.}}  Given a subset $X$ of a pro-finite group $G$, we write   $\overline X$ to denote the closure of $X$
in $G$.

\begin{proposition}
\label{correspondence} Let $\mathcal C$ be the class of finite nilpotent groups or finite $p$-groups (for a fixed prime $p$).
If $\G$ is a finitely generated discrete group which is residually 
$\mathcal C$, then
there is a one-to-one correspondence between the
set ${\mathcal X}$ of subgroups of $\G$ that are open in the pro-$\mathcal C$
topology on $\G$, and the set ${\mathcal Y}$ of all open subgroups
in the pro-$\mathcal C$ completion of $\G$.
 Identifying $\G$ with its image in the completion,
 this correspondence is given by:

\begin{itemize}

\item For $H\in {\mathcal X}$, $H \mapsto \overline{H}$.

\item For $Y\in {\mathcal Y}$, $Y\mapsto Y \cap \G$.\end{itemize}

\noindent If $H,K\in {\mathcal X}$ and $K<H$ then 
$[H : K] = [\overline{H}:\overline{K}]$. Moreover, $K\vartriangleleft H$
if and only if $\overline{K} \vartriangleleft \overline{H}$, 
and $\overline{H}/\overline{K} \cong H/K$.
\end{proposition}


\begin{corollary}
Let $\G$ be a finitely generated group that is residually nilpotent, and for each $d\in\mathbb \N$ let 
$M(d)$ denote the intersection of all the normal subgroups $\Delta \vartriangleleft\G$
of index $\le d$ such that $\G/\Delta$ is nilpotent. 
Let $\overline{M(d)}$ be the closure of
$M(d)$ in $\wh{\G}_\fn$. Then $\overline{M(d)}$ is the intersection of all the normal subgroups of
index $\le d$ in $\wh{\G}_\fn$, and hence $\bigcap_d \overline{M(d)} = 1$.
\end{corollary}

\begin{proof}  We begin with a preliminary remark.  It is proved in \cite{An} that
for a finitely generated group $\G$,
{\em every} subgroup of finite index in $\widehat{\G}_{\fn}$
is open, and likewise in $\wh{\G}_p$. (See \cite{NS} for the
case of a general profinite group.) Thus, by Proposition \ref{correspondence}, every normal subgroup of index $d$ in $\wh{\G}_{\fn}$
is the closure of a subgroup $K\vartriangleleft\G$ such that $\G/K$ is nilpotent and $|\G/K|=d$,
and every subgroup of index $d=p^k$ in $\wh{\G}_p$ is the closure of a subgroup of
index $p^k$ in  $\G$.

In the light of this remark, it   suffices to show
  that if $K_1,K_2<\G$ are normal and $Q_1=\G/K_1$ and $Q_2=\G/K_2$
  are finite and nilpotent, then $\overline{K_1\cap K_2} =
  \overline{K_1}\cap\overline{K_2}$.  But $\overline{K_1\cap K_2}$ is
  the kernel of the extension of $\G\to Q_1\times Q_2$ to
  $\widehat{\Gamma}_{\fn}$, while $\overline K_1\times \overline K_2$
  is the kernel of the map $\widehat{\Gamma}_{\fn}\to Q_1\times Q_2$
  that one gets by extending each of $\G\to Q_i$ and then taking the
  direct product.  These maps coincide on $\G$, which is dense, and
  are therefore equal.
\end{proof}

The same argument establishes:

\begin{corollary}\label{c:exhaust-p} Let $p$ be a prime and let
$\G$ be a finitely generated group that is residually p. For each $d\in\mathbb \N$ let 
$N(d)$ denote the intersection of all the normal subgroups 
$\Delta \vartriangleleft\G$
of $p$-power index less than $d$. Let $\overline{N(d)}$ be the closure of
$N(d)$ in $\wh{\G}_\fn$. Then $\overline{N(d)}$ is the intersection of all the normal subgroups of
index $\le d$ in $\wh{\G}_\fn$, and hence $\bigcap_d \overline{N(d)} = 1$.
\end{corollary}

\begin{remark} A key feature of the subgroups $M(d)$ [resp. $N(d)$] 
is that they form a fundamental system of open
neighbourhoods of $1\in\G$ defining the pro-(finite nilpotent) [resp. pro-$p$] topology. If
 we had merely taken an exhausting sequence of normal subgroups in $\G$, then
we would not have been able to conclude that the intersection in $\wh{\G}_{\fn}$ [resp. $\wh{\G}_{p}$] of their closures was trivial. 
\end{remark}
 
\subsection{Subgroups of $p$-Power Index}

 An advantage of the class of $p$-groups over the class of finite nilpotent groups is that the
former is closed under group extensions whereas the latter is not. The importance of this fact from our point of view is that 
it means that the induced topology on normal subgroups of $p$-power index $\Lambda<\G$ behaves well. 
The following is a consequence of Lemma 3.1.4(a) of \cite{RZ} 
(in the case where $\mathcal C$ is the class of finite $p$-groups).

\begin{lemma}\label{l:same}
Let $p$ be a prime and let $\G$ be a finitely generated group
that is residually-$p$. 
If $\Lambda<\G$ is a normal subgroup of $p$-power 
index, then the natural map $\wh{\Lambda}_p\to \overline{\Lambda}<\wh{\G}_p$ is an isomorphism.
\end{lemma}

\subsection{Betti numbers}
The first Betti number of a finitely generated group is
$$b_1(\Gamma) = \hbox{dim}_{\mathbb Q}\ [ (\Gamma/[\Gamma,\Gamma]) \otimes_{\mathbb Z} {\mathbb Q}].$$
Given any prime $p$, one can detect $b_1(\G)$  in the $p$-group quotients of $\G$, since
it is the {\em greatest integer $b$ such that $\G$ surjects $(\Z/p^k\Z)^b$ 
for every $k\in\N$.} We exploit this observation as follows:

\begin{lemma}\label{l:denseb1}
Let $\Lambda$ and $\G$ be finitely generated groups and let $p$ be a prime. If $\Lambda$ is isomorphic to a dense
subgroup of $\widehat{\G}_p$, then $b_1(\Lambda) \ge b_1(\Gamma)$.
\end{lemma}

\begin{proof} For every finite $p$-group $A$, each epimorphism
  $\widehat{\Gamma}\to A$ will restrict to an epimorphism on both
  $\Gamma$ and $\Lambda$ (since by density $\Lambda$ cannot be contained in a
  proper closed subgroup). But the resulting map
  $\Epi(\widehat{\Gamma},A)\to\Epi(\Lambda,A)$ need not be surjective,
  in contrast to Lemma \ref{NS}.  Thus if $\G$ surjects $(\Z/p^k\Z)^b$
  then so does $\Lambda$ (but perhaps not vice versa).
\end{proof}

\begin{corollary}\label{sameb1} If $\Lambda$ and $\G$ are finitely generated and $\widehat{\Lambda}_p\cong\widehat{\Gamma}_p$,
then $b_1(\Lambda)=b_1(\G)$.
\end{corollary}
 
\section{Criteria for pro-nilpotent equivalence}

The proof of the following proposition uses the
Lyndon-Hochschild-Serre (LHS) spectral sequence, which is explained on page 171 of \cite{brown}. 
Given a short exact sequence of groups $1\to N\to G\to Q\to 1$, the LHS spectral sequence calculates the homology of
$G$ in terms of the homology of $N$ and of $Q$. 
The terms on the $E^2$ page of the spectral sequence are
$E^2_{pq} = H_p(Q, H_q(N,\mathbb Z))$, where the action of $Q$ on $H_*(N,\mathbb Z)$
is induced by the action of $G$  on
$N$ by conjugation. 
The terms on the  diagonal $p+q=n$ of the $E^\infty$ page are the composition factors of a series for $H_n(G,\Z)$.

\begin{proposition}\label{p:GPnilp}  Let  $1\to N\overset{u}\to \G \to Q\to 1$ be a short exact sequence of groups
and let $u_c: N/N_c\to \G/\G_c$ be the homomorphism induced by  $u:N\hookrightarrow\G$.
Suppose that $N$ is finitely generated, that $Q$ has no non-trivial finite quotients, and that $H_2(Q,\Z)=0$.
Then  $u_c$ is an isomorphism for all $c\ge 1$, and hence  $\wh{u}_{\nil} : \wh{N}_{\nil} \to \wh{\G}_{\nil}$ is 
an isomorphism.
In particular, if $\G$ is residually nilpotent then $N$ and $\G$
have the same nilpotent genus.
\end{proposition}

\begin{proof} From the lower left corner of the LHS spectral sequence one obtains the following 5-term exact sequence (see
\cite{CE} page 328):
$$
H_2(Q,\Z) \to H_0(Q, H_1(N,\Z))\to H_1(\G,\Z)\to H_1(Q,\Z)\to 0.
$$
We have assumed that $H_2(Q,\Z)= H_1(Q,\Z)=0$, so the second arrow $H_0(Q, H_1(N,\Z))\to H_1(\G,\Z)$ is an isomorphism.

By definition, $H_0(Q, H_1(N,\Z))$ is the quotient of $H_1(N,\Z)$ by the action of $Q$. But since $H:= H_1(N,\Z)$ is a finitely
generated abelian group, ${\rm{Aut}}(H)$ is residually finite, whereas $Q$ has no non-trivial finite quotients. Thus,
the action of $Q$ on $H$ is trivial and the composition $H_1(N,\Z)\to H_0(Q, H_1(N,\Z))\to H_1(\G,\Z)$ (which is
the map on $H_1(-,\Z)$ induced by $u:N\to \G$) is an isomorphism.
Moreover, as $Q$ is perfect and $H_1(N,\Z)$ is a trivial $Q$-module, $H_1(Q,H_1(N,\Z)) =0$.
And since $E^2_{2,0} = H_2(Q, H_0(N,\Z)) = H_2(Q, \Z)$ is also assumed to be zero, the spectral sequence has only one
term  $E^2_{p,q}$ with $p+q=2$ that might be non-zero, namely $E^2_{0,2}=H_0(Q, H_2(N,\Z))$. It follows that the 
composition $H_2(N,\Z)\to H_0(Q, H_2(N,\Z))\to E^\infty_{0,2}\to H_2(\G,\Z)$ (which is
the map on $H_2(-,\Z)$ induced by $u:N\to \G$) is an epimorphism.

Stallings \cite{stall} proved that if a homomorphism of groups $u:N\to\G$ induces an isomorphism on $H_1(-,\Z)$ and an epimorphism
on $H_2(-,\Z)$, then $u_c: N/N_c\to \G/\G_c$  is an isomorphism for all $c\ge 1$.
\end{proof}

We remind the reader that the {\em fibre product} $P<\G\times \G$
associated to an epimorphism of groups $p:\G\epi Q$ is the subgroup
$P=\{(x,y)\mid p(x)=p(y)\}$.

\begin{corollary}\label{c:fpOK} Under the hypotheses of Proposition \ref{p:GPnilp},
the inclusion $P\hookrightarrow\G\times\G$ of the fibre product induces an isomorphism $P/P_c\to \Gamma/\G_c\times\Gamma/\G_c$ for every $c\ge 1$, and
hence an isomorphism $\wh{P}_{\nil} \to \wh{\G}_{\nil}\times\wh{\G}_{\nil}$.
\end{corollary}

\begin{proof} We have inclusions $N\times N\overset{i}\to P \overset{j}\to \G\times \G$.
The quotient of $P$ by $N\times N$ is $Q$ and the quotient of $\G\times\G$ by $N\times N$ is $Q\times Q$. The K\"unneth formula
tells us that $H_2(Q\times Q,\Z)=0$, and it is obvious that $Q\times Q$ has no finite quotients. Thus, by the proposition,
both $i$ and $j\circ i$ induce isomorphisms modulo any term of the lower central series, and therefore
$j$ does as well.
\end{proof} 
 
\subsection{Grothendieck Pairs and Pro-Nilpotent completions}

The criterion established in Proposition \ref{p:GPnilp} will be enough for the demands of Theorems A and B
but we should point out that (with different finiteness assumptions) one can prove a stronger result, namely that
$N\hookrightarrow\G$ and $P\hookrightarrow \G\times\G$ induce isomorphisms of profinite completions. 
The purpose of this subsection is to expand on this remark and explain why
 the statement about profinite completions is indeed stronger 
(cf.~Remark \ref{rr}).

Let $\G$ be a residually finite group and let $u:P\hookrightarrow \G$
be the inclusion of a subgroup $P$. Then $(\G,P)_u$ is called a {\em
  Grothendieck Pair} if the induced homomorphism $\widehat u:\widehat
P \to\widehat\G$ is an isomorphism but $u$ is not. (When no confusion is likely
to arise, it is usual to write $(\G,P)$ rather than $(\G,P)_u$.) Grothendieck
\cite{groth} asked about the existence of such pairs of finitely
presented groups and the first such pairs were constructed by Bridson
and Grunewald in \cite{BG}. The analogous problem for finitely
generated groups had been settled earlier by Platonov and Tavgen
\cite{PT}. Both constructions rely on versions
of the following result (cf.~\cite{PT}, \cite{BG} Theorem 5.1 and \cite{Bri}).

\begin{proposition}\label{l:pt}
Let $1\to N\to \G\to Q\to 1$ be a short exact sequence of groups with
$\G$ finitely generated and let $P$ be the associated fibre product.
 Suppose that $Q\neq 1$ is finitely presented, has
no proper subgroups of finite index, and $H_2(Q,\Z)=0$. Then
\begin{enumerate}
\item $(\G\times\G,P)$ is a Grothendieck pair;
\item if $N$ is finitely generated then $(\G,N)$ is a Grothendieck pair.
\end{enumerate}
\end{proposition}

Grothendieck pairs give rise to pro-nilpotent equivalences:

\begin{proposition}\label{groth-is-enough}
Let $u:P\hookrightarrow\G$ be a pair of finitely generated, residually
finite groups,
and for each $c\ge 1$, let $u_c: P/P_c\to \G/\G_c$ be the induced homomorphism.
If $(\G,P)_u$ is a Grothendieck pair, 
then $u_c$ is an isomorphism for all $c\ge 1$, and hence $\wh{u}_{\nil} : \wh{P}_{\nil} \to \wh{\G}_{\nil}$ is 
an isomorphism.
In particular, if $P$ and $\G$ are residually nilpotent, then they
have the same nilpotent genus.
\end{proposition}

Before presenting the proof of this proposition we make a remark and establish a lemma.

\begin{remark}\label{rr} In this proposition, it is vital that the isomorphism between $\wh{P}$ and $\wh{\G}$ is induced
by a map $P\to\G$ of discrete groups. For example, as we remarked earlier, there exist finitely generated nilpotent
groups that are not isomorphic but have the same profinite completion. A nilpotent group is  
its own pro-nilpotent completion, so these examples have the same profinite genus but different nilpotent genera.
\end{remark}
 
\begin{lemma}
\label{groth-bijection}
Let  $u:P\hookrightarrow\G$ be a pair of finitely generated, residually
finite groups. If  $(\G,P)_u$ is a Grothendieck pair, then for every finite group $G$, the
map  $q\mapsto q\circ u$
defines a bijection $\Epi(\G,G)\to \Epi(P,G)$.\end{lemma}

\begin{proof}  
As in Lemma \ref{NS}, by restricting homomorphisms $\wh{\G}\to G$ to $\G<\wh{\G}$ we obtain
a bijection  $\Epi(\wh{\G},G)\to \Epi(\G,G)$. Similarly, there is a bijection
$\Epi(\wh{P},G)\to \Epi(P,G)$. And the isomorphism $\wh{u}$ induces a bijection $\Epi(\wh{\G},G)\to \Epi(\wh{P},G)$.
Given $q\in \Epi(\G,G)$, the map $q\mapsto q\circ u$ described in Lemma \ref{NS} completes the commutative square
\[\xymatrix{
\Epi(\G,G)\ar@{>}[r]&\Epi(P,G)\\
\Epi(\wh{\G},G)\ar@{>}[r]\ar[u]&\Epi(\wh{P},G)\ar[u]
}\] and hence is a bijection.
\end{proof}

\noindent {\em Proof of Proposition \ref{groth-is-enough}:} If $G$ is nilpotent of class $c$, and $H$ is any group, then  
every homomorphism from $H$ to $G$ factors uniquely through $H/H_c$ and hence there is a natural
bijection  $\Epi(H/H_c, G)\to \Epi(H, G)$.  
By combining two such bijections with the epimorphism $q\mapsto q\circ u$ of Lemma \ref{groth-bijection},
$$
\Epi(\G/\G_c,G)\to \Epi(\G,G)\to \Epi(P,G)\to \Epi(P/P_c,G),
$$
we see that if $G$ is finite, then $q\mapsto q\circ u_c$ defines a bijection $\Epi(\G/\G_c,G)\to \Epi(P/P_c,G)$. 

Finitely generated nilpotent groups are residually
finite, so for every $c>0$ and every non-trivial element $\g\in P/P_c$, there is an epimorphism $\pi:P/P_c\to G$
to a finite  (nilpotent) group such that $\pi(\g)\neq 1$. The preceding argument provides $q\in \Epi(\G/\G_c,G)$
such that $\pi = q\circ u_c$, whence $u_c(\gamma)\neq 1$. Thus $u_c$ is injective.  

$u_c$ is also surjective, for if
$\g\in\G/\G_c$ were not in the image then using the subgroup
separability of nilpotent groups \cite{Mal}, we would have an
epimorphism $q:\G/\G_c\to G$ to some finite group such that $q(\gamma)\notin q\circ u_c
(P/P_c)$, contradicting the fact that $q\circ u_c$ is an epimorphism. \qed

\section{Conjugacy Problem: A Solution to Problem \ref{baumslag2}}

Recall that a finitely generated group $G$ is said to have a {\em solvable conjugacy problem} if there is an
algorithm that, given any pair of words in the generators, can correctly determine whether or
not these words define conjugate elements of the group. If no such algorithm exists then one says that the group
has an unsolvable conjugacy problem.  

In the light of Proposition \ref{groth-is-enough}, the following
theorem settles Problem \ref{baumslag2} (cf.~Remark \ref{r:Chuck1}).
 This theorem will be proved by combining the techniques of \cite{karl} with recent advances in 
the understanding of non-positively curved cube complexes and right-angled Artin groups (RAAGs). We remind the reader that a RAAG
is a group with a finite presentation of the form
$$
A = \langle a_1,\dots, a_n \mid [a_i,a_j]=1 \forall (i,j)\in E\rangle.
$$
Much is known about such groups. For our purposes here, their most important feature is that they are residually torsion-free-nilpotent  (\cite{DK}
Theorem 2.1).

\begin{theorem}
\label{solution1}
There exist pairs of finitely presented, residually
torsion-free-nilpotent groups $u:P\hookrightarrow\G$ such that $u$ induces isomorphisms of 
pro-nilpotent and profinite completions,  
but $\G$ has a solvable conjugacy problem while $P$ has an unsolvable
conjugacy problem.
\end{theorem}                         

\begin{proof} Let $1\to N\to H\stackrel{p}{\longrightarrow} Q\to 1$ be a
short exact sequence of groups and let $P< H\times H$ be the associated fibre
product. The  1-2-3 Theorem  \cite{bbms} states that if $N$ is finitely
generated, $H$ is finitely presented, and $Q$ has a classifying space
with a finite 3-skeleton, then the  $P$ is finitely presented
\cite{bbms}. 

A further result from \cite{bbms} states that if
$H$ is torsion-free and hyperbolic,
and $Q$ has an unsolvable word problem, then  $P<H\times H$ will have an unsolvable conjugacy problem. On the
other hand, the conjugacy problem in $H\times H$ is always solvable in polynomial time (cf. \cite{karl} \S 1.5).

Using ideas from \cite{collins-miller}, it is proved in \cite{karl}
that there exist groups $Q$ with no finite quotients that are of type
$F_3$, have $H_2(Q,\Z)=0$ and are such that the word problem in $Q$ is
unsolvable. 

Combining the conclusion of the preceding three paragraphs 
with Proposition \ref{l:pt}(1) (or Proposition \ref{p:GPnilp} in the pro-nilpotent case), we see that the theorem would be
proved if, given a finitely presented group $Q$ with the properties described above, we could construct a
short exact sequence $1\to N\to H\to Q\to 1$ with $N$ finitely
generated and with $H$ hyperbolic and residually
torsion-free-nilpotent: for then, $P\hookrightarrow H\times H$ would be the required pair
of groups.
 
In \cite{rips} Rips describes an algorithm that, given a finite presentation of any group
$Q$ will construct a short exact sequence $1\to K \to
H\stackrel{p}{\longrightarrow} Q\to 1$ where $K$ is  finitely generated and
$H$ is hyperbolic.
There are many refinements of Rips's construction
in the literature. Haglund and Wise \cite{HW} proved a version in which  
$H$ is constructed to be {\em virtually special}. By definition,
a virtually special group $H$ has a subgroup of finite index $H_0<H$ that is a
subgroup of a RAAG \cite{HW}, and as remarked upon above,
RAAGs are residually torsion-free-nilpotent by \cite{DK}. 
Consider the short exact sequence $1\to K\cap H_0\to H_0\to p(H_0)\to 1$. 
If we take $Q$ to be as in third paragraph of the proof, then 
$p(H_0)=Q$, because $Q$ has no proper subgroups of finite index.
 And $N=K\cap H_0$, being of finite index in $K$, is
finitely generated. Thus we have constructed a short exact sequence
$1\to N\to H_0\to Q\to 1$ of the required form.
\end{proof}

\begin{remarks} 
\begin{enumerate}
\item In the preceding proof we quoted Haglund and Wise's version of the Rips construction; this
was used to ensure that the group $H$ was virtually special and hence virtually residually torsion-free-nilpotent.
Less directly, it follows from recent work of Agol \cite{agol} and Wise \cite{wise} that the groups produced by the original Rips
 construction \cite{rips} are also virtually special.

\item If one is willing to reduce the finiteness properties in the
  above theorem then one can simplify the proof: it is easy to prove
  that if $Q$ has an unsolvable word problem and $p:H\to Q$ is an
  epimorphism from a hyperbolic group with ${\rm{ker}}\, p$ finitely
  generated, then the conjugacy problem in ${\rm{ker}}\, p$ is
  unsolvable. So if $1\to N\to H\to Q\to 1$ is constructed as in the
  above proof, then $(H,N)$ is a Grothendieck pair of groups in the
  same nilpotent genus, but $H$ is hyperbolic and $N$ has an
  unsolvable conjugacy problem.
\end{enumerate}
\end{remarks}

\begin{remark}\label{r:Chuck1} Commenting on an earlier version of this manuscript,  Chuck Miller pointed out
to us that his work with Baumslag on the isomorphism problem
for residually torsion-free nilpotent groups \cite{BaM} contains an implicit
solution to Problem 1.1.  To prove the
main theorem in that paper, the authors construct a sequence of finite presentations for residually torsion-free nilpotent
groups $G_w$ indexed by words in the generators of an auxillary group $H$ that has an unsolvable word problem.
$G_w$ is isomorphic to $G_1$ if and only if $w=1$ in $H$. If $w\neq 1$ in $H$ then 
$G_w$  has unsolvable conjugacy problem, whereas $G_1$ has a solvable conjugacy problem.  
It follows that some $G_w$ with $w\neq 1$ must have the same nilpotent genus as $G_1$, for if not then one could
decide which $G_w$ were isomorphic to $G_1$ be running the following algorithm.

 Given a finite presentation of $\G_w$, use the solution to the
isomorphism problem for nilpotent groups to test if the quotients of $G_w$ by the terms
of its lower central series are isomorphic to the corresponding quotients of $G_1$. (Explicit presentations for these quotients
are obtained by simply adding basic commutators to the given presentations of $G_w$ and $G_1$.)
 If the nilpotent genus of $G_w$ is different from
that of $G_1$, this process will eventually find non-isomorphic quotients. At the same time, search naively
for an isomorphism from $G_w$ to $G_1$. These parallel processes will together
 determine whether or not $G_w$ is isomorphic to
$G_1$ {\em unless}  $G_w$ and $G_1$ are non-isomorphic groups of the same nilpotent genus.

Notice that although this argument settles Problem 1.2, it does not prove Theorem A(2): it does not construct
an explicit pair of groups $G_1$ and $G_{w\neq 1}$ with the same nilpotent quotients, it
just proves that  such pairs exist. Moreover, there is no homomorphism between the discrete groups inducing the
isomorphism of pro-nilpotent completions.
\end{remark}

\section{Finite Presentation: A Solution to Problem \ref{baumslag1}}\label{s:b1}

We present two constructions settling Problem \ref{baumslag1}. Both rely on the existence of finitely presented infinite groups $Q$ with
$H_2(Q,\Z)=0$ that have no non-trivial finite quotients. A general method for constructing such groups
is described in \cite{BG}. The first such group was discovered by Graham Higman:
$$
Q = \langle a,b,c,d \mid bab^{-1}=a^2,\ cbc^{-1}=b^2,\ dcd^{-1}=c^2,\ ada^{-1}=d^2\rangle.
$$

\subsection{Subdirect products of free groups}

Finitely generated free groups are residually torsion-free
nilpotent \cite{Mag0}, and hence so are subgroups of their direct
products. Thus the following proposition resolves Problem \ref{baumslag1}. 
 
\begin{proposition}
\label{Plat-Tav}
Let $Q$ be a finitely presented infinite group with $H_2(Q,\Z)=0$
that has no finite quotients. Let $F$ be a finitely generated free
group, let $F\to Q$ be a surjection, and let $P< F\times F$ be
the associated fibre product. Then:
\begin{enumerate}
\item $P$ is finitely generated but not finitely presented;
\item $P\hookrightarrow F\times F$ induces an isomorphism of pro-nilpotent completions. 
\end{enumerate}
\end{proposition}
 
 \begin{proof} Let $F$ be the free group on $\{x_1,\dots,x_n\}$. Since $Q$ is finitely
presented, the kernel of $F\to Q$ is the normal closure of a finite set $\{r_1,\dots,r_m\}$.
It is easy to check that $P < F\times F$ is generated by $\{(x_1,x_1),\dots,(x_n,x_n),
(r_1,1),\dots,(r_m,1)\}$. But Grunewald \cite{fritz} proves that $P$ is  finitely
presentable if and only if $Q$ is finite. Assertion   (2) is a special case of
 Corollary \ref{c:fpOK}.
 \end{proof}
 
 \begin{remark} Platonov and Tavgen \cite{PT} proved that $P\hookrightarrow F\times F$ also induces an isomorphism of profinite completions. 
 \end{remark}

\subsection{A solution via the Rips construction}

Again, let $Q$ be a finitely presented infinite group with $H_2(Q,\Z)=0$ that has no non-trivial finite quotients.
In the proof of Theorem \ref{solution1} we used a version of the Rips construction
to obtain  a short exact sequence $1\to N \to H\to Q\to 1$ with $H$
  hyperbolic (in particular finitely presented) and residually torsion-free
  nilpotent, and with $N$ finitely generated and not free. A further
  property of the Rips construction is that $H$ is small cancellation, in particular of cohomological dimension $2$.
  Bieri \cite{Bi} proved that a finitely presented normal
  subgroup of infinite index in a group of cohomological dimension $2$
  must be free.  Thus if $Q$ is infinite then $N$ is not finitely
  presented. But Proposition \ref{p:GPnilp} tells us that $N$ and $\G$ have the same nilpotent genus (and
  Proposition \ref{l:pt} tells us that $(N,\G)$ is a Grothendieck pair). 

\section{Nilpotent Genus and the Schur Multiplier: 
Problem \ref{baumslag3}}

In the previous section we settled the question of whether there exist finitely generated residually
torsion-free-nilpotent groups of the same genus such that one is
finitely presented and the other was not. In \cite{chat} Baumslag laid particular emphasis on 
a homological variant of this question: is there is a pair of
finitely generated residually torsion-free-nilpotent groups of the
same genus such that the Schur multiplier $H_2(G,\Z)$ of one is
finitely generated, while that of the other group is not?  We shall
settle this question by proving the following theorem. Here, and in
what follows, we take homology with coefficients in $\mathbb Q$. When
we discuss the {\em{dimension}} of a homology group, we mean its
dimension as $\mathbb Q$-vector space.  Note that if $H_n(G,\Z)$ is
finitely generated then $H_n(G,\mathbb Q)$ is finite dimensional.

\begin{theorem} \label{t:H2'}
There exists  a pair of finitely generated residually
torsion-free-nilpotent groups $N\hookrightarrow \G$ that have the same nilpotent genus and the same profinite
completion, but  $\G$ is finitely presented while $H_2(N,\mathbb Q)$ is infinite dimensional.
\end{theorem} 

Our proof of the above theorem draws on the ideas in previous sections, augmenting them with a spectral sequence argument 
to control the homology of $N$. The proof also relies on the construction of a group with particular properties that we regard
as having independent interest -- see Proposition \ref{p:newD}.

\subsection{A spectral sequence argument}

\begin{proposition}\label{p:ss} Let $1\to N\to G \to Q\to 1$ be a short exact sequence of finitely generated groups.
If $H_3(G,\mathbb Q)$ is finite dimensional  but $H_3(Q,\mathbb Q)$ is infinite dimensional, then  
$H_2(N,\mathbb Q)$ is infinite dimensional.
\end{proposition}

\begin{proof} As in the proof of Proposition \ref{p:GPnilp} we use the
LHS spectral sequence to calculate the homology of
$G$ in terms of $N$ and $Q$. The terms on the $E^2$ page of
the spectral sequence are
$E^2_{pq} = H_p(Q, H_q(N,\mathbb Q))$.

$H_3(Q, H_0(N,\mathbb Q))$ is infinite dimensional and $H_1(Q, H_1(N,\mathbb Q))$ is finite dimensional,
so $E^3_{30}$, which is the kernel of $d_2: E^2_{30}\to E^2_{11}$, is infinite dimensional. 

The kernel of $d_3 : E^3_{30}\to E^3_{02}$ is 
$E^\infty_{30}$, which is a section of $H_3(G,\mathbb Q)$ and hence  is finite dimensional.
So the image of the map to  $E^3_{02}$ is infinite dimensional. But  $E^3_{02}$
is a quotient of  $H_0(Q, H_2(N,\mathbb Q))$, which in turn is a quotient of $H_2(N,\mathbb Q)$.
Thus  $H_2(N,\mathbb Q)$ is infinite dimensional. 
\end{proof}

\subsection{A designer group} 

Recall that a group $A$ is termed {\em acyclic} (over $\Z$) if $H_i(A,\Z)=0$ for all $i\ge 1$. The Higman group
described in Section \ref{s:b1} was the first example of a finitely presented acyclic group with no
proper subgroups of finite index. Further examples were constructed in \cite{BG}, including, for each integer $p\ge 3$,
$$ 
\langle a_1,a_2, b_1, b_2 \mid a_1^{-1}a_2^pa_1 a_2^{-p-1},\,
 b_1^{-1}b_2^pb_1 b_2^{-p-1},\,
a_1^{-1}[b_2, b_1^{-1}b_2b_1],\, b_1^{-1}[a_2, a_1^{-1}a_2a_1]\rangle.
$$

Let $A$ be one of the above groups. The salient features of $A$ are that  it is
finitely presented,  acyclic over $\Z$, has no finite quotients,  
 contains a 2-generator free group, $F$ say, and is torsion-free (indeed it has a 
2-dimensional classifying space $K(A,1)$, cf.~\cite{BG} p.364). Let $\Delta = (A\times A) \ast_S (A\times A) $
 be the double of $A\times A$ along $S<F\times F$, where $S$ is the
first Stallings-Bieri group, i.e. the kernel of a homomorphism $F\times F\to \Z$ whose restriction to each of the factors is surjective.
The key features of $S$
are that it  is finitely generated but $H_2(S,\mathbb Q)$ is infinite dimensional (see \cite{stall}, or \cite{BH} pp. 482-485).

\begin{lemma}\label{delta} $\Delta$ is torsion-free, finitely presented, has no non-trivial finite quotients,
and  $H_3(\Delta, \mathbb Q)$ is infinite dimensional.
\end{lemma}

\begin{proof} The amalgamated free product of two finitely presented groups along a finitely generated subgroup
is finitely presented, so $\Delta$ is finitely presented. And an amalgam of torsion-free groups is 
torsion-free. The four visible copies of $A$ generate $\Delta$, and these
all have trivial image in any finite quotient, so $\Delta$ has no non-trivial finite quotients.
 We calculate $H_3(\Delta, \mathbb Q)$ using the Mayer-Vietoris sequence
(omitting the coefficient module $\mathbb Q$ from the notation): 
$$
\dots H_3(A\times A)\oplus H_3(A\times A)\to H_3(\Delta) \to H_2(S) \to H_2(A\times A)\oplus H_2(A\times A)\to\dots
$$
As $A$ is acyclic, so is $A\times A$, by the K\"unneth formula. Hence $H_3(\Delta, \mathbb Q) \cong H_2(S, \mathbb Q)$
is infinite dimensional.
\end{proof} 

Recall that a group $G$ is termed {\em super-perfect} if $H_1(G,\Z)=H_2(G,\Z)=0$. Proposition \ref{l:pt}
explains our interest in this condition.   $\Delta$ is perfect  but it is not super-perfect.

\begin{lemma}\label{H2D} $H_2(\Delta,\Z)\cong H_1(S,\Z)\cong \Z^3$.
\end{lemma}

\begin{proof} A slight variant of the above Mayer-Vietoris argument
shows that $H_2(\Delta,\Z)\cong H_1(S,\Z)$. 

The first homology of $S$ can be calculated using Theorem A of \cite{BM}:
if $G\le F\times F$ is a subdirect product and we write $F_1=F\times 1$ and $L=G\cap F_1$, then 
$$H_1(G,\Z)\cong H_1(F,\Z) \oplus H_2(F_1/L,\Z)\oplus C,$$
where $C={\ker} (H_1(F_1,\Z)\to H_1(F_1/L, \Z))$. 

In our case, $G=S$ and $F_1/L=\Z$, so $H_2(F_1/L,\Z)=0$  and $C$ is cyclic. (Recall that $F\cong F_1$ is free of rank $2$.)
\end{proof}

To obtain a super-perfect group, we pass
to the {\em universal central extension} $\widetilde\Delta$.

\begin{proposition}\label{p:newD}
 $\tilde\Delta$ is  torsion-free, finitely presented, super-perfect, has no non-trivial finite quotients,
and  $H_3(\tilde\Delta, \mathbb Q)$ is infinite dimensional.
\end{proposition}

\begin{proof} The standard theory of universal central extensions (see \cite{milnor}, Chapter 5) tells us that $\tilde\Delta$
is perfect and that there is a
short exact sequence
$$
1\to H_2(\Delta, \Z)\to \tilde \Delta \to \Delta \to 1.
$$
Since $\Delta$ and $H_2(\Delta, \Z)$ are finitely presented and torsion-free, so is  $\tilde\Delta$. Since $\Delta$ has no non-trivial
finite quotients, $H_2(\Delta, \Z)$ would have to map onto any finite quotient of $\tilde\Delta$, which means that
all such quotients are abelian. Since $\tilde\Delta$ is perfect, it follows that it has no non-trivial finite
quotients (cf.~\cite{BG}, p.369).

To see that  $H_3(\tilde\Delta, \mathbb Q)$ is infinite dimensional, we consider the LHS spectral sequence associated to the
above short exact sequence. As $\Delta$ is finitely presented and $K:=H_2(\Delta, \Z)$ is of type ${\rm{FP}}_\infty$, all of
the groups in the first three columns, $E^2_{pq}=H_p(\Delta, H_q(K,\mathbb Q))$ with $0\le p \le 2$, are finite dimensional.
On the other hand, $E^2_{30}=H_3(\Delta, \mathbb Q)$ is infinite dimensional. Therefore $E^3_{30}=
\ker (E^2_{30}\to E^2_{11})$ and $E^\infty_{30}=E^4_{30} = \ker (E^3_{30}\to E^3_{02})$ are infinite dimensional.
But $E^\infty_{30}$ is a quotient of $H_3(\tilde\Delta, \mathbb Q)$, so $H_3(\tilde\Delta, \mathbb Q)$ is also infinite dimensional.
\end{proof}

\subsection{Proof of Theorem \ref{t:H2'}}

We have constructed a group $\tilde\Delta$ that is super-perfect, finitely presented  and has $H_3(\tilde\Delta,\mathbb Q)$
infinite dimensional. By applying a suitable version of the Rips construction to $\tilde\Delta$ (as in the proof of Theorem \ref{solution1}),
we obtain a short exact sequence
$$
1\to N \to H \to \tilde\Delta \to 1
$$
with $N$ finitely generated and $H$ is a 2-dimensional hyperbolic group that is
 virtually special. Passing to a subgroup of finite index $H_0<H$ and replacing $N$
by $N\cap H_0$, we may assume that $H$ is a subgroup of a RAAG, and 
hence is residually torsion-free-nilpotent. Propositions \ref{p:GPnilp} and \ref{l:pt}(2)
tell us that $N\to H$ induces an isomorphism of pro-nilpotent and profinite completions.
Proposition \ref{p:ss} tells us that $H_2(N,\mathbb Q)$ is infinite dimensional.
\qed
 
\begin{remark}\label{r:Chuck2} We outline an alternative proof 
of Theorem \ref{t:H2'} suggested by Chuck Miller. Let $A$ be a group satisfying the
conditions of Section 6.2.
We fix a finite presentation $A=\<X\mid R\>$ then
augment if by adding a  new generator $x_0$ and a new relation $x_0=1$.
Let $F$ be the free group on $X\cup\{x_0\}$, let $F\to A$ be the natural epimorphism, and let $P<F\times F$ be
the associated fibre product. The results in Section 3 tell us that $P$ and $F\times F$ have the same
nilpotent quotients.
Because $(x_0,1)\in P$ is part of a free basis for $F\times \{1\}$, one can express
$P$ as an HNN extension with stable letter $x_0$ and amalgamated subgroup
$P\cap (\{1\}\times F)$. Then, as in the proof of Theorem 4 of \cite{cfm:qjm}, one can use the
Mayer-Vietoris sequence for the HNN extension to prove that 
$H_2(P,\Z)$ maps onto 
the first homology of the infinitely generated free group ${\rm{ker}}(F\to A)$.
\end{remark}

\section{Normal subgroups of parafree groups}

We settled Baumslag's first three questions by constructing groups of a somewhat pathological nature that
lie in the same nilpotent genus as well-behaved groups. Our solution to 
Problem \ref{baumslag4} is of an entirely different nature: the point here is to prove that parafree groups
share a significant property with free groups.
Correspondingly, the nature of the mathematics that we shall draw on is entirely different.

Our proof of the following theorem relies on a mix
of $L^2$-Betti numbers and profinite group theory that we first
employed in our paper \cite{BCR} with M. Conder.  

\begin{theorem}(=Theorem \ref{t:N})
\label{maininfindex}
If $\Gamma$ is a finitely generated parafree group, then
every finitely generated non-trivial normal subgroup of $\G$ is of finite index.\end{theorem}

\subsection{$L^2$-Betti numbers} The standard reference for this material is L\"uck's treatise \cite{Lu2}.
In what follows $b_1(X)$ denotes the usual first Betti number of a group, and $b_1^{(2)}$ denotes
the first $L^2$-Betti number. We shall not recall the definition of the $L^2$-Betti number as it does not
inform our arguments. 
In the case of finitely presented groups, one can use
{\em L\"{u}ck's Approximation Theorem} \cite{Lu1} to give a surrogate definition of $b_1^{(2)}$: Suppose that
$\G$ is {\em finitely presented} and let
$$\G = N_1 > N_2 > \ldots > N_m > \ldots,$$
\noindent be a sequence of normal subgroups, each of finite index in  $\G$, with $\bigcap_m N_m =1$; L\"uck proves that
$$\lim_{m\rightarrow \infty} 
        \frac{b_1(N_m)}{[\G:N_m]} = b_1^{(2)}(\G).$$

\begin{example}\label{bF} Let $F$ be a free group of rank $r$. Euler characteristic tells us that a subgroup of index $d$ in $F$ is
free of rank $d(r-1) +1$, so by L\"{u}ck's Theorem $b_1^{(2)}(F_r)= r-1$. A similar calculation shows that if $\Sigma$ is the
fundamental group of a closed surface of genus $g$, then $b_1^{(2)}(\Sigma) =2g-2$.
\end{example}      
        
If one assumes only that the group $\G$ is {\em finitely generated}, then one does not know if the above limit exists, and
when it does exist one does not know if it is independent of the chosen tower of subgroups. This is a problem in the context of
Theorem \ref{maininfindex} because we do not know if finitely generated parafree groups are finitely presentable. 
Thus we appeal instead to the weaker form of L\"uck's
approximation theorem established for finitely generated groups by L\"uck and Osin \cite{LO}.

\begin{theorem}
\label{luck}
If $\G$ is a finitely generated residually finite group and $(N_m)$ is a  sequence of finite-index
normal subgroups with $\bigcap_m N_m =1$, then
$$\limsup_{m\rightarrow \infty}
        \frac{b_1(N_m)}{[\G:N_m]} \leq b_1^{(2)}(\G).$$
\end{theorem}

Another result about $L^2$-Betti numbers that we will make use of is
the following theorem of Gaboriau (see \cite{Gab} Theorem 6.8).

\begin{theorem}
\label{gab}
Suppose that 
$$1\rightarrow N\rightarrow \Gamma\rightarrow \Lambda\rightarrow 1$$
is an exact sequence of groups where $N$ and $\Lambda$ are infinite. 
If $b_1^{(2)}(N)<\infty$, then $b_1^{(2)}(\Gamma)=0$.\end{theorem}

\subsection{$L^2$-Betti numbers of dense subgroups}

The key step in the proof of Theorem \ref{maininfindex} is the following result (cf. \cite{BCR} Proposition 3.2). Recall that we write $\wh{\G}_p$ to denote the pro-$p$ completion of a
group $\G$, and $\overline{H}$ to denote the closure of a subgroup in $\wh{\G}_p$.

\begin{proposition}
\label{boundL_2}
Let $\G$ be a finitely generated group and let $F$ be a finitely presented group
that is residually-$p$ for some prime $p$. Suppose that 
there is an injection $\G\hookrightarrow \widehat{F}_{p}$
and that $\overline{\G} = \widehat{F}_{p}$.  Then
$b_1^{(2)}(\G) \geq b_1^{(2)}(F)$.\end{proposition}

\begin{proof} For each positive integer $d$ let $N(d)< F$ be the
  intersection of all normal subgroups of $p$-power index at most $d$ in $F$.
  Let $L(d) = \G\cap\overline{N(d)}<\widehat{F}_p$.  We saw
  in Corollary \ref{c:exhaust-p} that $\bigcap_d \overline N(d) = 1$,
  hence $\bigcap_d L(d)=1$.  Since $\G$ and $F$ are both dense in
  $\widehat{F}_p$, the restriction of $\widehat{F}_p\to
  \widehat{F}_p/\overline{N(d)}$ to each of these subgroups is
  surjective, and therefore
$$[\G : L(d)] = [\widehat{F}_p: \overline{N(d)}]= [F: N(d)].$$ 

$L(d)$ is dense in $\overline{N(d)}$ and in Lemma \ref{l:same} we saw
that $\widehat{N(d)}_p \cong \overline{N(d)}$, so Lemma
\ref{l:denseb1} implies that $b_1(L(d)) \ge b_1(N(d))$.  We use the
towers $(L(d))$ in $\G$ and $(N(d))$ in $F$ to compare $L^2$-betti
numbers, applying Theorem \ref{luck} to the finitely generated group
$\G$ and L\"uck's Approximation Theorem to the finitely presented
group $\G$:
$$b_1^{(2)}(\G)\geq \displaystyle \limsup\limits_{d\rightarrow \infty}
{{{b_1(L(d))}\over{[\G:L(d)]}} }\geq\displaystyle \limsup\limits_{d\rightarrow \infty}
{{{b_1(N(d))}\over{[F:N(d)]}} }=\displaystyle \lim\limits_{d\rightarrow \infty} {{{b_1(N(d))}\over{[F:N(d)]}} }=
b_1^{(2)}(F)$$ 
as required.
\end{proof} 

\begin{corollary}\label{c:sameL2}
Let $\Lambda$ and $\G$ be finitely presented groups that are residually-$p$ for some prime $p$.
If $\widehat{\Gamma}_p\cong\widehat{\Lambda}_p$ then 
$b_1^{(2)}(\G) = b_1^{(2)}(\Lambda) $.
\end{corollary}

By combining Proposition \ref{boundL_2} with Theorem \ref{gab} we deduce:

\begin{theorem}\label{t:general} Let $\G$ be a finitely generated group and let $N\vartriangleleft\G$ be a non-trivial normal
subgroup. Let $F$ be a finitely presented group that is residually-$p$
for some prime $p$ and suppose that there is an injection 
$\G\hookrightarrow \wh{F}_p$ with dense image.
If $b_1^{(2)}(F)>0$, then either $b_1^{(2)}(N) =\infty$ or else $|\G/N|<\infty$. In particular, 
if $N$ is finitely generated then it is of finite index.
\end{theorem}

\bigskip

\noindent{\bf{Proof of Theorem \ref{maininfindex}.}} 
Suppose now that $\G$ is a finitely generated parafree group, with the
same nilpotent genus as the free group $F$, say. Finitely
generated groups with the same nilpotent genus have isomorphic pro-$p$
completions for every prime $p$ (cf.~Remark 2.1(9)), so
 $\wh{\G}_p\cong\wh{F}_p$. 
 
 If $F$ is cyclic, then it is
easy to see that $\Gamma$ must also be cyclic, so the conclusion of
Theorem \ref{maininfindex} holds. Therefore we assume that $F$ has rank
$r>1$.  

$\G$ is residually-$p$ for all primes $p$. Indeed, since
$\G$ is residually nilpotent, every non-trivial element of $\G$ has a
non-trivial image in $\G/\G_c\cong F/F_c$ for some term $\G_c$ of the
lower central series, and the free nilpotent group $F/F_c$ is
residually-$p$ for all primes $p$ (see \cite{Gru}).  Combining this
observation with the conclusion of the first paragraph, we
see that the natural map from $\G$ to $\wh{\G}_p\cong \wh{F}_p$ is
injective. In Example \ref{bF} we showed that $b_1^{(2)}(F_r)=(r-1)>0$.
Thus we are in the situation of Theorem \ref{t:general} and the proof
is complete.  \qed
   
\section{Applications} 

We close with some further consequences of the arguments in the
preceding section, and a discussion of related matters. 

\subsection{Parafree groups are $C^*$-simple}

For ease of exposition, it will be convenient to
introduce the following definition.
A group $\G$ is said to be {\em parafree of pf-rank $r$} if 
$\G$ is residually nilpotent and is in the same nilpotent genus as a 
free group of rank $r$.  As  
special cases of Proposition
\ref{boundL_2} and Corollary \ref{c:sameL2} we have:

\begin{corollary}
\label{L2parafree}
If $\G$ is a finitely generated parafree group of  pf-rank $r$, then
$$b_1^{(2)}(\G) \ge r-1=b_1(\G)-1,$$
with equality if $\G$ is finitely presented. 
\end{corollary}

A group
$\Gamma$ is {\em $C^*$-simple} if its
{\it reduced $C^*$-algebra} $C_\lambda^*(\Gamma)$
 is simple as a complex algebra (i.e.
has no proper two-sided ideals).  
By
definition, the {\em reduced $C^*$-algebra} $C^*_\lambda(\G)$ is the norm closure of the image
of the complex group algebra $\Bbb C [\G]$ under the left-regular
representation $\lambda_\G : \Bbb C [\G]\to \mathcal L (\ell^2(\Gamma))$
defined for $\gamma\in\G$
by $\left(\lambda_{\Gamma}(\gamma)\xi\right)(x) =
\xi(\gamma^{-1}x)$
for all $x \in \Gamma$ and $\xi \in \ell^2(\Gamma)$. A group
$\Gamma$ is $C^*$-simple if and only if 
any unitary representation of $\Gamma$ which is weakly contained in
$\lambda_{\Gamma}$ is weakly equivalent to $\lambda_{\Gamma}$. We refer the reader to \cite{dlH} for a thorough
account of the groups that were known to be {\em $C^*$-simple} by 2006. The subsequent work of 
Peterson and Thom \cite{PeT} augments this knowledge.

An important early result in the field is the proof by Powers \cite{powers} that 
non-abelian free-groups are $C^*$-simple. 
 In contexts where one is able to adapt the Powers argument,
one also expects the canonical trace to be the only normalized trace on $C^*_\lambda(\G)$ (cf. Appendix to \cite{BdlH}).
By definition,
a linear form $\tau$ on $C^*_\lambda(\Gamma)$ is a {\it normalised trace}
if $\tau(1)=1$ and $\tau(U^*U)\ge 0,\ \tau(UV)=\tau(VU)$ 
for all $U,V \in C^*_\lambda(\Gamma)$.
The {\it canonical trace} is uniquely defined by
$$
\tau_{\text{\rm{can}}}\left(\sum_{f\in F}z_f\lambda_\Gamma(f)\right) 
\, = \, z_e
$$
for every {\it finite} sum $\sum_{f\in F}z_f\lambda_\Gamma(f) $ where
$z_f\in\Bbb C$ and $F \subset \Gamma$ contains $1$.

\begin{corollary}
\label{C*simple}
Let $\G$ be a finitely generated group. If $\G$
is  parafree of pf-rank $r\geq 2$, then the reduced group $C^*$-algebra $C^*_\lambda(\G)$
is simple and carries a unique normalised trace.\end{corollary}

\begin{proof} Corollary \ref{L2parafree} tells us that $b_1^{(2)}(\G)\neq 0$. By definition,
$\G$ is residually torsion-free nilpotent, and therefore it satisfies condition $(\star)$
of \cite{PeT}, Section 4, i.e.~every non-trivial element of $\Z[\G]$ acts without kernel on $\ell^2(\G)$.
With these facts in hand, the present result follows immediately from
Corollary 4.6 of \cite{PeT} (which in turn relies on \cite{BCH}).\end{proof}

\subsection{A Freiheitssatz for parafree groups}
 
Section 4 of Peterson and Thom \cite{PeT} contains a number of other results concerning the structure
of finitely presented groups that satisfy their condition $(\star)$ and have non-zero  $b_1^{(2)}$.
In the context of parafree groups, the following consequence of
our results and  \cite{PeT} has a particular appeal as it 
it gives a new proof of a result of Baumslag \cite{Baum69} (Theorem 4.1). 

\begin{corollary}
\label{freiheitssatz}
Let $\Gamma$ be a finitely generated group that is parafree
of pf-rank $r\geq 2$. Then every 
generating set $S\subset \G$ has an $r$-element subset $T\subset S$ such that
the subgroup of $G$ generated by
$T$ is free  of rank $r$. 
\end{corollary}

\begin{proof} With Corollary
\ref{L2parafree} in hand, we can apply  Corollary 4.7 of \cite{PeT} where the conclusion of the corollary 
is shown to hold
for any finitely generated group $\G$ with $b_1^{(2)}(\G) > r-2$ that satisfies $(\star)$.
\end{proof}

This result strengthens greatly Magnus's observation that an $r$-generator
group of parafree rank $r$ must be free of rank $r$.
Also, the proof shows that if $\G$ is an $(r+1)$-generator group of 
parafree rank $r$ (not necessarily finitely presented)
then $b_1^{(2)}(\G)=r-1$ (cf.~Corollary \ref{L2parafree}). 

\begin{remark} The {\em{parafree conjecture}} posits that $H_2(G,\Z)=0$ for every finitely generated parafree group $G$. 
If this were proved, then  Corollary \ref{freiheitssatz} would be a consequence of a theorem of Stallings \cite{St} which
states that if $G$ is a finitely generated group with $H_2(G,\Z)=0$ and $H_1(G,\Z)=\Z^r$ then any $Y\subset G$ that is
independent in $H_1(G,\Q)$ freely generates a free subgroup.
\end{remark}

\subsection{Parafree groups and lattices} 

\begin{corollary}
\label{notlattice}
Let $\G$ be a group of parafree rank $r\geq 2$ that
is not free. Then
$\G$ is not isomorphic to a lattice in a connected Lie group.\end{corollary}

\begin{proof} If $\G$ were a lattice then it
would be finitely generated, so $b_1^{(2)}(\G)\neq 0$ by Corollary \ref{L2parafree}.
It is shown in \cite{Lo} that if $\G$ is a lattice in a connected
Lie group and  $b_1^{(2)}(\G)\neq 0$, then $\G$ is commensurable with
a lattice in $\PSL(2,\R)$; i.e. the group is virtually free or
virtually the fundamental group of a closed orientable surface
of genus at least $2$.  Now $\G$ is torsion-free, with torsion-free
abelianization, so in fact, in this
case, $\G$ is free or the fundamental group $\Sigma_g$ of a closed orientable surface
of genus $g\ge 2$. The former is ruled out by assumption, and the latter
is ruled out by the observation that $b_1^{(2)}(\Sigma_g) = 2g-2 =b_1(\Sigma_g)-2$;
see Example \ref{bF}.
(Alternatively, it is straightforward to construct a finite nilpotent 
quotient of a free group of rank $2g$ that cannot be a quotient
of  the genus $g$ surface group.)\end{proof} 
\subsection{Homology boundary links.}
 In contrast to Corollary \ref{notlattice}, we give an example of a lattice in
$\PSL(2,\mathbb C)$ that {\em does} have the same lower central series as a 
free group.  However, the lattice is not residually nilpotent.

\begin{example} Consider  the link $L$ 
shown below. The complement of this link is hyperbolic,
$S^3\smallsetminus L \cong {\mathbb H}^3/\Gamma$ where $\Gamma<\PSL(2,\mathbb C)$
is the torsion-free non-uniform lattice denoted  A2 in \cite{BFLW}. 

\begin{center}
\vbox{\epsfysize=1.5truein\epsfbox{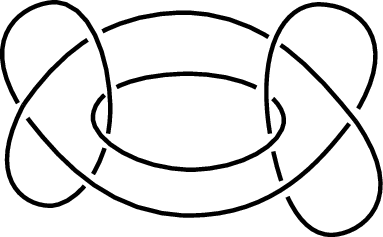}
\\
}
\end{center}
\bigskip

From \cite{BFLW}, we have the presentation
$$\G=\<u,v,z,l\mid [u,l]=1,uzu^{-1}=v^{-1}zvz, l=v^{-1}uzu^{-1}vz\>,$$
where $u$ is a meridian for the unknotted component and $l$
is a longitude for $u$.  The other peripheral subgroup (i.e. the one 
corresponding to the square knot component) is 
$\<uzu^{-1},uv^2uv^{-1}\>$, with meridian $uv^2uv^{-1}$ and longitude $uzu^{-1}$.
The two meridians generate $H_1(\G,\Z)=\Z^2$.

By setting $z=1$ we obtain a homomorphism from $\G$ onto $\<u,v,l\mid l=1 \>$; i.e. the free group of rank $2$.
This homomorphism  $\G\rightarrow F_2$ induces an isomorphism
$H_1G\to H_1F_2$ and a surjection on higher homology groups
(since $H_k(F_2,\Z)=0$ for all $k\geq 2$). An application
of Stallings' Theorem \cite{St} now establishes that $\G$ has the same lower central series quotients as $F_2$.

Note that the homomorphism $\G\rightarrow F_2$ can be realised topologically by
performing $0$-surgery on both the unknotted component (forcing $l=1$) and  the knotted
component (forcing $uzu^{-1}=1$, hence
$z=1$).  The closed manifold resulting from these surgeries is a connected sum of two copies of $\mathbb{S}^2\times\mathbb{S}^1$.

Corollary \ref{notlattice} (alternatively, Remark \ref{r:ref1}) tells us that $\Gamma$ is not residually nilpotent. In fact, it is not difficult to see
that the longitudes described
above lie in the intersection of the terms in the lower central series of $\G$.\qed
\end{example}

\begin{remark}\label{r:ref1} An anonymous referee pointed out to us that the lattice described above exemplifies the following
more general phenomenon. Let $F$ be a free group, let $G$ be a finitely generated group that is not free, and suppose that
there is an epimorphism  $\pi:G\to F$ that induces an isomorphism $\overline\pi:H_1(G,\Z)\to H_1(F,\Z)$. Then $G$ is not residually nilpotent.
To see that this is the case, note that since $F$ is free there is a homomorphism $\sigma:F\to G$ such $\pi\circ\sigma ={\rm{id}}_F$. 
The map that $\sigma$ induces on abelianisations is the inverse of $\overline\pi$; in particular $\sigma(F)$ generates $G/[G,G]$.
But if a set generates a nilpotent group modulo the commutator subgroup, then it generates the nilpotent group itself.
It follows that $\sigma\circ\pi$ induces an isomorphism $F/F_c\to G/G_c$ for all $c>0$. Therefore,
$\ker\, \pi$ (which is assumed to be non-trivial) is contained in the intersection of the terms of the lower central series of $G$.  
\end{remark}

The link depicted above is an example of a homology boundary link.
Recall that a link $L\subset S^3$ of $m\geq 2$
components is called a {\em homology boundary link} if there exists an
epimorphism $h:\pi_1(S^3\setminus L)\rightarrow F$ where $F$ is a free
group of rank $m$.  By Alexander duality, $H_1(\pi_1(S^3\smallsetminus L),\Z)\cong\Z^m$.
Thus $h$ induces an isomorphism
$H_1(\pi_1(S^3\smallsetminus L),\Z)\cong H_1(F_m,\Z)$, and
arguing as above we conclude that (with the exception of trivial links) 
the fundamental groups of homology boundary links are never residually nilpotent. 
On the other hand, standard results guarantee that all lattices in $\PSL(2,\mathbb C)$ 
are {\em virtually} residually $p$ for all but a finite number of primes $p$
(and hence virtually residually nilpotent). Recent work of Agol \cite{agol}
implies that these lattices are also virtually residually torsion-free-nilpotent.

\subsection{Pro-$p$ goodness}
 Homology boundary links also provide interesting
examples in the following related context.  

One says that a group $\G$ is {\em pro-p good} 
if for each $n\geq 0$, the homomorphism of cohomology groups
$$H^n(\widehat{\G}_p;{\F}_p)\rightarrow H^n(\G;{\F}_p)$$
induced by the natural map $\G\rightarrow \widehat{\G}_p$ is an isomorphism,
where the group on the left is in the continuous cohomology of 
$\widehat{\G}_p$.
One says that the group $\G$ is  {\em cohomologically complete}
if $\G$ is pro-$p$ good for all primes $p$.  

It is shown in \cite{BLS}
that many link groups are cohomologically complete, and indeed
it was claimed by Hillman, Matei and Morishita \cite{HMM}
that all link groups are cohomologically complete. 
Counterexamples to the method of \cite{HMM} were given in \cite{Blo}, 
but \cite{Blo} left open the possibility that link groups might nevertheless be cohomologically complete.  Here we note that
homology boundary links provide counterexamples. In particular there are hyperbolic links that provide counterexamples.

\begin{proposition}
\label{nocohomcom}
Let $L$ be a homology boundary link that is not the trivial link of $m$ components.
Then $\pi_1(S^3\smallsetminus L)$ is not pro-$p$ good for any 
prime $p$. In particular, $\pi_1(S^3\smallsetminus L)$ is not cohomologically
complete.\end{proposition}

\begin{proof} Let $\G=\pi_1(S^3\smallsetminus L)$. Since $\G$ has the same nilpotent quotients as the free group of rank $m$, it follows 
 that $\wh{\G}_p$ is isomorphic
to a free pro-$p$ group. Hence $H^2(\wh{\G}_p;{\F}_p)=0$.
On the other hand, since $L$ is a non-trivial link with $m\geq 2$
components,  $H^2(\G;{\F}_p)$ has dimension $m-1>0$ as
an ${\F}_p$-vector space. \end{proof}

By way of contrast, we also note that the lattice $\G$ described in \S 8.2 is
{\em good} in the sense of Serre; i.e.
for each $n\geq 0$ and for every finite $\G$-module $M$, the homomorphism of cohomology groups
$$H^n(\widehat{\G};M)\rightarrow H^n(\G;M)$$
induced by the natural map $\G\rightarrow \widehat{\G}$ is an isomorphism
between the cohomology of $\G$ and the continuous 
cohomology of $\widehat{\G}$. (Note that goodness deals with the profinite completion and not the pro-$p$ completions.)
The goodness of $\G$ follows from
\cite{GJZ} since $\Gamma$ is a subgroup of finite index in a Bianchi group.

\subsection{} Theorem \ref{t:general} 
can be usefully applied to cases where $F$ is not free; for example, 
 $F$ might be a non-abelian surface group or, more
generally, a non-abelian limit group. That these satisfy the condition on
the first $L^2$-Betti number can be seen in Example 7.2 
for surface groups and \cite{BK} for limit groups. This leads to an
analogue of Theorem 7.1 for {\em paralimit} groups.

Note also that, for both surface groups and limit groups, condition $(\star)$ of
\cite{PeT} holds (since these groups are residually torsion-free
nilpotent).
Hence, groups that are paralimit or parasurface groups, in the above sense, also
satisfy a version of Freiheitssatz as in Corollary \ref{freiheitssatz}.

The notion of a ``parasurface group'' was considered in \cite{BR}.

\end{document}